\documentclass[11pt,a4paper]{article}
\usepackage[utf8]{inputenc}

\usepackage[english]{babel}
\usepackage[shortcuts]{extdash}
\usepackage[blocks]{authblk}

\usepackage{amsmath}
\usepackage{amssymb}
\usepackage{enumitem}

\mathcode`l="8000
\begingroup
\makeatletter
\lccode`\~=`\l
\DeclareMathSymbol{\lsb@l}{\mathalpha}{letters}{`l}
\lowercase{\gdef~{\ifnum\the\mathgroup=\m@ne \ell \else \lsb@l \fi}}%
\endgroup

\makeatletter
\let\@fnsymbol\@arabic
\makeatother

\usepackage[scale=.73]{geometry}

\usepackage{xcolor}
\usepackage{tikz}
\usetikzlibrary{calc}

\tikzset{
	vertex/.style={circle,fill=black,draw,minimum size=5pt,inner sep=0pt},
	edgeS/.style={line width=.7pt},
	edgeD/.style={line width=.7pt,dashed},
}

\usepackage[sortcites, style=math,citestyle=numeric-comp, giveninits=true,dashed=false]{biblatex}
\usepackage[thmmarks,hyperref]{ntheorem}
\theoremstyle{plain}
\theoremheaderfont{\normalfont\bfseries}
\theorembodyfont{\slshape}
\theoremseparator{.}
\theorempreskip{3ex}
\theorempostskip{\topsep}
\theoremnumbering{arabic}
\theoremsymbol{}
\newtheorem{theorem}{Theorem}

\newtheorem{corollary}[theorem]{Corollary}

\newtheorem{proposition}[theorem]{Proposition}

\makeatletter
\newtheoremstyle{proof}%
{\item[\rlap{\vbox{\hbox{\hskip\labelsep \theorem@headerfont
##1\theorem@separator}\hbox{\strut}}}]}%
{\item[\rlap{\vbox{\hbox{\hskip\labelsep \theorem@headerfont ##1\ ##3\theorem@separator}\hbox{\strut}}}]}
\makeatother

\theoremheaderfont{\normalfont\bfseries}
\theorembodyfont{\upshape}

\theoremstyle{proof}

\theorempreskip{3.5ex}
\theorempostskip{\topsep}
\theoremseparator{.}

\theoremsymbol{\ensuremath{_\Box}}

\newtheorem{proof}{Proof}

\theoremheaderfont{\itshape}
\theoremsymbol{\ensuremath{_\Diamond}}

\makeatletter
\newcommand\need[1]{\par \penalty-100 \begingroup %
   \dimen@\pagegoal \advance\dimen@-\pagetotal %
   \ifdim #1>\dimen@ %
      \ifdim\dimen@>\z@ \vskip -\pagedepth plus 1fil \fi
      \break
   \fi \endgroup}
\makeatother
\newcommand\newparagraph{\vspace*{3ex}}

\usepackage{hyperref}
\hypersetup{
    colorlinks=true,
    linkcolor=black,
    citecolor=black,
    pdftitle={A Construction of Uniquely Colourable Graphs with Equal Colour Class Sizes},
    pdfauthor={Samuel Mohr},
    bookmarks=true,
    pdfpagemode=UseOutlines,
}

\title{A Construction of Uniquely Colourable Graphs with Equal Colour Class Sizes}

\author{Samuel Mohr\thanks{Gefördert durch die Deutsche Forschungsgemeinschaft (DFG) -- 327533333.}}

\affil{Institut für Mathematik der Technischen Universität Ilmenau, Weimarer Straße 25, 98693 Ilmenau, Germany}

\begin{document}

\maketitle

\thispagestyle{plain}

\begin{abstract}
\setlength{\parindent}{0em}
\setlength{\parskip}{1.5ex}
\noindent

A \emph{uniquely $k$\=/colourable} graph is a graph with exactly one partition of the vertex set into at most $k$ colour classes. 
Here, we investigate some constructions of uniquely $k$\=/colourable graphs and give a construction of $K_k$\=/free uniquely $k$\=/colourable graphs with equal colour class sizes. 

\medskip



\end{abstract}

We use the standard terminology of graph theory and consider \emph{simple}, \emph{finite} graphs $G$ with vertex set $V(G)$ and edge set $E(G)$. 
A \emph{$k$\=/colouring} of a graph $G$ with $k\in\mathbb{N}$ is a partition $\mathcal{C}$ of the vertex set $V(G)$ into $k'\leq k$ non\-/empty sets $A_1,\dots,A_{k'}$. 
The colouring $\mathcal{C}$ is called \emph{proper} if each set is an independent set of $G$, that means that there are no two adjacent vertices of $G$ in the same \emph{colour class} $A\in\mathcal{C}$. 
The \emph{chromatic number} $\chi(G)$ is the minimum $k$ such that there is a proper $k$\=/colouring of $G$.

We call a graph $G$ \emph{uniquely $k$\=/colourable} if $\chi(G)=k$ and for any two proper $k$\=/colourings $\mathcal{C}$ and $\mathcal{C}'$ of $G$, we have $\mathcal{C}=\mathcal{C}'$. 
It is easy to see that the complete graph $K_k$ on $k$ vertices is uniquely $k$\=/colourable and we can obtain a family of uniquely $k$\=/colourable graphs by consecutively adding a vertex to the vertex set of a uniquely $k$\=/colourable
graph and joining it to all vertices except those of one colour class. 
This raises the question if all uniquely $k$\=/colourable graphs contain $K_k$ as a subgraph. 

\newparagraph

The properties of uniquely colourable graphs have been widely studied, for example in \cite{chartrand1969uniquely,harary1969uniquely,shaoji1990size,akbari2001kr,nevsetvril1973uniquely,bollobas1976uniquely,emden1998uniquely}. 
One such property---to be found in \cite{chartrand1969uniquely}---is that the union of any two distinct colour classes induces a connected graph. 
Assume to the contrary that there is a graph $G$ with unique colouring $\mathcal{C}$ and there are $A,B\in\mathcal{C}$, $A\neq B$,  such that $G[A\cup B]$ has at least two components. 
Let $H$ be such a component and consider the colouring $\tilde{\mathcal{C}}$ with $\tilde{\mathcal{C}}=(\mathcal{C}\setminus\{A,B\}) \cup \{(A\setminus V(H)) \cup (B\cap V(H))\} \cup \{(B\setminus V(H)) \cup (A\cap V(H))\}$. 
The new partition  $\tilde{\mathcal{C}}$ is obtained from the
original partition interchanging $A\cap V(H)$ and $B\cap V(H)$.
Then $\tilde{\mathcal{C}}$ is a proper colouring distinct from $\mathcal{C}$, a contradiction. 
We say $\tilde{\mathcal{C}}$ is obtained from $\mathcal{C}$ by a \emph{Kempe change along $H$}. 

This implies that in a uniquely $k$\=/colourable graph every vertex has a neighbour in every other colour class. Hence, a uniquely $k$\=/colourable graph is connected and has minimum degree at least $k-1$. 
Furthermore, a uniquely $k$\=/colourable graph is $(k-1)$-connected. 
To see this, assume that there is a non\-/complete  graph $G$ with a unique $k$\=/colouring $\mathcal{C}$ and for two non\-/adjacent vertices $x,y$, there is a separator $S$ with $|S|\leq k-2$ such that $x$ and $y$ are in distinct components of $G- S$. 
But then there are distinct $A,B\in \mathcal{C}$ with $A\cap S=\emptyset=B\cap S$ and $(G-S)[A\cup B]=G[A\cup B]$ is connected. Since $x$ and $y$ have neighbours in $A\cup B$, they cannot be separated by $S$, a contradiction. 

\newparagraph

This question whether a uniquely $k$\=/colourable graphs always contains $K_k$ as a subgraph was first disproved by Harary, Hedetniemi, and Robinson~\cite{harary1969uniquely} in 1969. 
They gave a uniquely 3\=/colourable graph $F$ without triangles. For $k\geq 4$, an example of a uniquely $k$\=/colourable graph without $K_k$ is $F+K_{k-3}$, where $G_1+G_2$ is the complete join of  graphs $G_1$ and $G_2$. 

Several years later, Xu~\cite{shaoji1990size} proved that the number of edges of a uniquely $k$\=/colourable graph on $n$ vertices is at least $(k-1)\,n-\binom{k}{2}$ and that this is best possible. 
He further conjectured that uniquely $k$\=/colourable graphs with exactly this number of edges have $K_k$ as a subgraph~\cite{shaoji1990size}. 
This conjecture was disproved by Akbari, Mirrokni, and Sadjad~\cite{akbari2001kr}. They constructed a $K_3$\=/free uniquely 3\=/colourable graph $G$ on 24 vertices and 45 edges. For the cases of $k\geq 4$, again $G+K_{k-3}$ disproves the conjecture. 

We are interested in constructions of uniquely $k$\=/colourable graphs  such that the colour classes have “nearly the same size”. 
One useful concept for this is the critical chromatic number introduced by Komlós~\cite{komlos2000tiling} in the context of bounds on a Tiling Turán number. 
Given a $k$\=/colourable graph $H$  on $h$ vertices, let $\sigma(H)$ be the smallest possible size of a colour class in any proper $k$\=/colouring of $H$. 
Then the  \emph{critical chromatic number} is defined by
\begin{align*}
\chi_{cr}(H)=(\chi(H)-1)\cdot\frac{h}{h-\sigma(H)}. 
\end{align*}
The critical chromatic number fulfils $\chi(H)-1<\chi_{cr}(H)\leq \chi(H)$ and equality holds if and only if in every $k$\=/colouring of $H$ the colour classes have exactly the same size. 

All constructions above produce graphs with critical chromatic number tending to $\chi(H)-1=k-1$.

\newparagraph

In the following, we give a new construction of uniquely $k$\=/colourable graphs. 
Given a uniquely $k$\=/colourable graph $H$ without $K_k$ and $\chi_{cr}(H)=\chi(H)$, this construction leads to a uniquely $(k+1)$\=/colourable graph $G$ without $K_{k+1}$ and $\chi_{cr}(G)=\chi(G)$. 
We further compare this construction with a result of Nešetřil~\cite{nevsetvril1973uniquely} and a probabilistic proof for the existence of uniquely colourable graphs by Bollobás and Sauer in~\cite{bollobas1976uniquely}.

\newparagraph

\centerline{*}

\newparagraph

\subsubsection*{Construction}

Let $H$ be a $k$\=/colourable graph with a proper $k$\=/colouring $\mathcal{C}=\{A_1,\dots,A_k\}$.
We obtain a graph $G=\nu(H)$ with a proper $(k+1)$\=/colouring $\mathcal{C}'$ by adding $k$ copies of $H$ to $H$ (for each colour class of $\mathcal{C}$ one copy), joining each copy to the original graph by edges in the same way as in the lexicographic product of $H$ and $\overline{K_2}$, and finally inserting edges such that the $k$ copies of each vertex of $H$ induce a star with central vertex in the copy of $H$ corresponding to the colour class of the original vertex in $H$. 

Thus $G=\nu(H)$ with a proper $(k+1)$\=/colouring $\mathcal{C}'$ consists of
\begin{align*}
V(G)=V(H)&\cup\{v^p\mid v\in V(H), p=1,\dots ,k\}, \\
E(G)=E(H)&\cup\{v^pu^p\mid vu\in E(H), p=1,\dots ,k\}\\
&\cup \{vu^p\mid vu\in E(H), p=1,\dots ,k\}\\
&\cup \{v^pv^q\mid v\in A_p, q\in\{1,\dots ,k\}\setminus\{p\}\}
\end{align*} 
and 
\begin{align*}
\mathcal{C}'=\{A_i'  \mid i=1,\dots,k\}\cup\{\{v^p \mid v\in A_p,p=1,\dots,k\}\}
\end{align*}
with $A_i'=\{v,v^p \mid v\in A_i,p\in\{1,\dots ,k\}\setminus\{i\}\}$.

\begin{theorem}\label{thm}
Let $H$ be a uniquely $k$\=/colourable graph with $k\geq 3$, then $\nu(H)$ is uniquely $(k+1)$\=/colourable.  
\end{theorem}

\begin{proof}
First, it is straightforward to check that $\mathcal{C}'$ is a proper colouring with $k+1$ colours. 
Therefore, assume that there is another colouring $\tilde{\mathcal{C}}$ with $k+1$ colours. 

Fix a colour class $D\in\tilde{\mathcal{C}}$ with $V(H)\cap D\neq\emptyset$. 
If $v\in V(H)\cap D$, vertices $\{v^1,\dots ,v^k\}$ belong  to at least two distinct colour classes of $\mathcal{C}'$. 
Thus, there is an index $p_v\in\{1,\dots,k\}$ such that $v^{p_v}\notin D$. 
Let a mapping $f^D:V(H)\to V(G)$ be defined by $f^D(v)=v$ if $v\in V(H)\setminus D$ and $f^D(v)=v^{p_v}$ otherwise, and  $X(D)=\{f^D(v)\mid v\in V(H)\}$. 
Then $f^D$ represents an isomorphism between $H$ and $G[X(D)]$, and the colouring $\tilde{\mathcal{C}}$ induces a $k$\=/colouring of $G[X(D)]$, which is unique (by the unique
$k$\=/colourability of $H$). 
Therefore $f^D(v)$ and $f^D(w)$ with $v,w\in V(H)$ are in the same colour class of $\tilde{\mathcal{C}}\setminus \{D\}$ if and only if 
$f^C(v)$ and $f^C(w)$  are in the same colour class of $\tilde{\mathcal{C}}\setminus \{C\}$ for any colour class  $C\in\tilde{\mathcal{C}}\setminus \{D\}$ satisfying $V(H)\cap C\neq\emptyset$. 

Assume first that $V(H)\cap D\neq \emptyset$ for all $D\in\tilde{\mathcal{C}}$. 
Let us show that $f^D$ maps $V(H)\cap D$ to a subset of one colour class of $\tilde{\mathcal{C}}\setminus \{D\}$.
Indeed, otherwise, there were $v,w\in V(H)\cap D$, $v\neq w$ and $A,B\in \tilde{\mathcal{C}}\setminus\{D\}$, $A\neq B$ such that  $f^D(v)\in A$ and $f^D(w)\in B$.
Since $|\tilde{\mathcal{C}}|=k+1\geq 4$, there is $C\in\tilde{\mathcal{C}}\setminus\{A,B,D\}$. 
Then both $f^C(v)$ and $f^C(w)$ are in $D$, while  $f^D(v)\in A\neq B\ni f^D(w)$, a contradiction. 
Thus, for each $D\in\tilde{\mathcal{C}}$ there is a colour class $A^D\in\tilde{\mathcal{C}}$ such that $f^D(v)\in A^D$ for all $v\in V(H)\cap D$. 

According to our assumption with fixed $D\in\tilde{\mathcal{C}}$ we have $V(H)\cap A^D\neq \emptyset$. 
Then for each $B\in\tilde{\mathcal{C}}\setminus\{A^D,D\}$, the induced subgraph $G[X(D)\cap (A^D\cup B)]$ is connected. 
So there is a vertex $w\in V(H)\cap B$ such that $N_H(w)\cap A^D\neq \emptyset$.
This vertex $w$ has a neighbour in all colours classes except $B$, and, consequently, the same is true for each $w^i$ with $i\in\{1,\dots,k\}$.
Hence, $\{w,w^1,\dots ,w^k\}\subseteq B$, a contradiction. 

Therefore, the vertices of $H$ belong to $k$ colour classes of $\tilde{\mathcal{C}}$. Then, because of the unique $k$\=/colourability of $H$, 
$\{A\cap V(H)\mid A\in\tilde{\mathcal{C}}, A\cap V(H)\neq \emptyset\}=\mathcal{C}$. Let $D\in\tilde{\mathcal{C}}$ be the colour class satisfying $D\cap V(H)=\emptyset$. 
If $v\in V(H)$ and $A\in\tilde{\mathcal{C}}\setminus \{D\}$ are such that $v\in A$, vertices $\{v^1,\dots ,v^k\}$ belong  to the two colour classes $A,D\in\mathcal{C}'$ (since $N_G(v^i)\cap V(H)=N_H(v)$ for all $i\in\{1,\dots ,k\}$). 
Moreover, they induce a star $S_v$ in $G$ and let $v^*\in\{v^1,\dots ,v^k\}$ be the central vertex of $S_v$, i.\,e.\ $N_{S_v}(v^*)>1$. 
Assume first that there is $v\in V(H)$ such that $v^*\notin D$, then $\{v^1,\dots ,v^k\}\setminus\{v^*\}\subseteq D$. 
Let $w\in V(H)$ be a neighbour of $v$ with $w\in B\in\tilde{\mathcal{C}}$. 
With $v^*=v^p$ and $w^*=w^q$, then $\{w^1,\dots ,w^k\}\setminus \{w^p\}\subseteq B$. 
Since clearly $p\neq q$, for $r\in\{1,\dots ,k\}\setminus\{p,q\}$ we have $\{w^*,w^r\}\subseteq B$, 
which is in contradiction to $w^*w^r\in E(G)$.
Therefore, we can conclude that  $v^*\in D$ for all $v\in V(H)$, it follows $\tilde{\mathcal{C}}=\mathcal{C}'$; and the theorem is proved. 
\end{proof}

\newparagraph

\begin{proposition}\label{prop:properties}
The graph  $G=\nu(H)$ for a  $k$\=/colourable graph $H$ on $n$ vertices  satisfies:
\begin{enumerate}[label=\emph{\alph*})]
\item $G$ is uniquely $(k+1)$\=/colourable if $k\geq 3$ and $H$ is uniquely $k$\=/colourable,\label{enum:unique}
\item $\omega(G)=\omega(H)+1$, \label{enum:omega}
\item $\chi_{cr}(G)=\chi(G)$ if $\chi_{cr}(H)=\chi(H)$,\label{enum:critical}
\item $|E(G)| = (3k + 1)\,|E(H)| + (k - 1)\,n$ and $|V(G)| = (k + 1)\,n$,\label{enum:size}
\item The minimum degree of $G$ is $2\,\delta(H)+1$,\label{enum:delta}
\item $H$ is an induced subgraph of $G$. \label{enum:induced}
\end{enumerate}
\end{proposition}

\begin{proof}
By Theorem~\ref{thm} we showed \ref{enum:unique} and it is easy to see \ref{enum:induced} from the construction. Simply counting vertices and edges leads to \ref{enum:size}.

To show \ref{enum:omega}, let $C_H$ be a maximum clique of $H$ (cliques are vertex sets). 
If $v\in C_H\cap A_p$ for some $p\in \{1, \dots, k\}$, then with $q\in \{1,\dots , k\}\setminus \{p\}$ the set $(C_H\setminus\{v\})\cup\{v^p,v^q\}$ is a clique in $G$, 
hence $\omega(G)\geq |C_H|+1=\omega(H)+1$. 

Now let $C$ be a clique in $G$ with $C \nsubseteq V(H)$. 
Assume that there are distinct $r,s,t\in\{1,\dots,k\}$ and $u,v,w\in V(H)$ such that $u^r,v^s,w^t\in C$. 
Since these three vertices induce a triangle, it is $u=v=w$; however, $v^r,v^s,v^t$ can only induce a path in $G$, a contradiction. 
If there are distinct $r,s\in\{1,\dots,k\}$ and $u,v\in V(H)$ with $u^r,v^s\in C$, then $u = v$ and $(C \setminus \{v^r, v^s \}) \cup \{v\}$ is a clique in $H$, which implies
$|C| \leq  \omega(H) + 1$. 

Thus there is $p\in\{1,\dots,k\}$ such that $C\setminus V(H)\subseteq\{v^p\mid v\in V(H)\}$. 
Since $|C\cap \{v, v_p \}| \leq  1$ for each $v\in V(H)$, the set $(C\cap V(H))\cup \{v\mid v^p\in C\}$ induces a clique in $H$ of the size $|C| \leq \omega(H)$, which is not maximum one. 
Hence we conclude $\omega(G) = \omega(H) + 1$ and \ref{enum:omega} is proved. 

If $\chi_{cr}(H)=\chi(H)$, then each colour class of $H$ has size $s=\frac{|V(H)|}{\chi(H)}$. 
By the construction of the colouring $\mathcal{C}'$, each colour class of $G$ is of size $k s$. 
Since $\mathcal{C}'$ is unique, we get $\chi_{cr}(G)=\chi(G)$ and we have proved \ref{enum:critical}. 

By the construction, we obtain the following degree function for $v\in V(G)$, which shows \ref{enum:delta}. 
\begin{align*}
d_{G}(x)=\begin{cases}(k+1)\,d_H(v), &\text{if $x=v\in V(H)$,}\\2\,d_H(v)+k-1, &\text{if $x=v^p$ and $v\in A_p$,}\\2\,d_H(v)+1, &\text{if $x=v^p$ and $v\notin A_p$.}\end{cases}
\end{align*}
\end{proof}

\subsubsection*{Small  triangle\-/free uniquely 3\=/colourable graphs}

Using some computer calculation, Figure~\ref{fig1} shows a graph on 12 vertices and 22 solid drawn edges. 
Adding at most one of the dashed edges, we obtain one  of three non\-/isomorphic uniquely 3\=/colourable triangle\-/free graphs on 12 vertices with critical chromatic number 3. 

\begin{figure}
\begin{center}%
\begin{minipage}{.7\linewidth}%
\begin{center}%
\begin{tikzpicture}[scale=1.1]

\foreach \v in {2,5}  {
\node[vertex,fill=blue] (v\v) at (45*\v+22.5:1.08) {};
}
\foreach \v in {1,4,7}  {
\node[vertex,fill=red] (v\v) at (45*\v+22.5:1.08) {};
}
\foreach \v in {0,3,6}  {
\node[vertex,fill=green] (v\v) at (45*\v+22.5:1.08) {};
}

\node[vertex,fill=blue] (a) at (2,1) {};
\node[vertex,fill=red] (c) at (-2,-1) {};
\node[vertex,fill=blue] (b) at (2,-1) {};
\node[vertex,fill=green] (d) at (-2,1) {};

\foreach \x[evaluate=\x as \y using {int(Mod(\x+1,8))}] in {0,...,7} {
\draw[edgeS] (v\x)--(v\y);
}
\foreach \x[evaluate=\x as \y using {int(\x+4)}] in {0,...,3} {
\draw[edgeS] (v\x)--(v\y);
}

\draw[edgeS] (a)--(v1);
\draw[edgeS] (a)--(v7);
\draw[edgeS] (b)--(v0);
\draw[edgeS] (b)--(v6);
\draw[edgeS] (c)--(v5);
\draw[edgeS] (c)--(v3);
\draw[edgeS] (d)--(v2);
\draw[edgeD] (d)--(v4);

\draw[edgeS] (d) to[out=20,in=160] (a);
\draw[edgeS] (c) to[out=-20,in=-160] (b);
\draw[edgeS] (c)--(d);

\draw[edgeD] (v4) to[out=40,in=-170] (a);
\end{tikzpicture}%
\end{center}%
\vspace{-2ex}
\caption{Non\-/isomorphic uniquely 3\=/colourable triangle\-/free graphs on 12 vertices with critical chromatic number 3}\label{fig1}
\end{minipage}%
\end{center}
\end{figure}

\begin{corollary}\label{cor:constr}
For all $k\geq 3$ there are uniquely $k$\=/colourable $K_k$\=/free graphs on $2k!$ vertices with critical chromatic number $k$. 
\end{corollary}

\begin{proof}
For $k=3$, it is straightforward to check that the graphs in Figure~\ref{fig1} are uniquely 3\=/colourable, triangle\-/free, have 12 vertices and critical chromatic number 3. 

For $k > 3$, if $G$ is one of three graphs described above, then, by Proposition~\ref{prop:properties}, the graph $\nu^{k-3}(G)$ has all properties demanded by Corollary~\ref{cor:constr}.
\end{proof}

Probably the graphs in Figure~\ref{fig1} are the only non\-/isomorphic uniquely 3\=/colourable triangle\-/free graphs with critical chromatic number 3 on 12 vertices. 
To support this conjecture independent computer calculations would be
useful. 
For the moment we leave the conjecture open, and we conclude the
section by showing that the graphs in Figure~\ref{fig1} have the smallest possible
number of vertices.

\begin{proposition}
The graphs in Figure~\ref{fig1} have the smallest number of vertices among all uniquely 3\=/colourable triangle\-/free graphs with critical chromatic number 3. 
\end{proposition}

\begin{proof}
Since the critical chromatic number is 3, the number of vertices $n$ of such graphs has to be divisible by 3. 
We leave the case $n=3$ and $n=6$ to the reader and assume that there is such a graph $G$ with $n=9$ vertices. 

As mentioned above, each two distinct colour classes of the unique 3\=/colouring $\mathcal{C}$ of $G$ induce a connected subgraph of $G$.
Therefore let $A,B\in\mathcal{C}$ be distinct colour classes, then $|A|=|B|=3$ and $G[A\cup B]$ connected. 
If there was a vertex $v\in A$ with degree~3 in $G[A\cup B]$, then each neighbour $w\in C$ of $v$ in the third colour class $C$ would form a triangle with $v$ and a suitable vertex from $B$. 
Hence $G[A\cup B]$ is either a path or a cycle on six vertices. 
In both cases, at least two vertices in $A$ have degree~2 in $G[A\cup B]$. By a symmetric argument, there are two  vertices in $A$ having degree~2 in $G[A\cup C]$ and  two  vertices in $B$ having degree~2 in $G[B\cup C]$. 
Thus, there is a vertex $v\in A$ with a neighbour $w\in B$ such that there is a common neighbour $u\in C$ of $v$ and $w$; and we obtain a triangle on $\{u,v,w\}$, a contradiction. 
\end{proof}

\subsubsection*{Comparison with other results}

In this section, we compare our new construction of uniquely colourable graphs  with a straightforward construction, with a triangle\-/free construction by Nešetřil~\cite{nevsetvril1973uniquely}, 
and a  probabilistic proof  of existence of uniquely colourable graphs of arbitrary girth by Bollobás and Sauer~\cite{bollobas1976uniquely}. 
To obtain, for an integer $k\geq 3$, a uniquely $k$\=/colourable $K_k$\=/free graphs with equal colour class sizes, none of the following constructions is suitable.
To our best knowledge, our construction is the first one to produce such
graphs; 
thus, the construction~$\nu$ and Corollary~\ref{cor:constr} seem to fill in a gap.

\newcounter{constr}

\newparagraph
\refstepcounter{constr}
\need{2cm}
\noindent\textit{Construction~\theconstr: }\label{constr1}

The complete graph $K_k$ on $k$ vertices is uniquely $k$\=/colourable. 

Given a uniquely $k$\=/colourable graph $H$, adding to it a new vertex adjacent
to all vertices of $H$ except for those of one colour class, we obtain another uniquely $k$\=/colourable graph. 
If we start with $H$ satisfying $\chi_{cr}(H) < k$, then
choosing in the above construction repeatedly a smallest colour class yields a
sequence of uniquely $k$\=/colourable graphs with increasing critical chromatic
number.

However, if the original graph has a unique maximum colour class, then, after balancing all
colour class sizes as described above, we finally obtain a $K_k$ in the resulting
graph.
Hence, depending on our starting graph $H$, in some cases we cannot obtain a uniquely $k$\=/colourable $K_k$\=/free graph with equal colour class sizes using Construction~\ref{constr1}. 

\newparagraph
\refstepcounter{constr}
\need{2cm}
\noindent\textit{Construction~\theconstr\ (Nešetřil~\cite{nevsetvril1973uniquely}): }\label{constr2}

To get a uniquely $k$\=/colourable graph, $k\geq 3$, choose $n>16\,k\cdot(2(k-2))^{2k-1}$ and start with the uniquely $2$\=/colourable path $P_n^0=P_n$ on $n$ vertices and colour classes $A_1,A_2$. 
The uniquely $k$\=/colourable graph $P_n^{k-2}$ is constructed iteratively. 
Assume that $P_n^{j-1}$, $j\geq 1$ with colour classes $A_1,\dots,A_{j+1}$ is constructed and let $\mathcal{M}^j$ be the set of all independent sets $M$ of  $P_n^{j-1}$ with $|M|=j+2$ such that $M\cap A_i\neq \emptyset$, $1\leq i\leq j+1$. 
Then $V(P_n^{j})=V(P_n^{j-1})\cup \mathcal{M}^j$ and $xy\in E(P_n^{j})$ if $xy\in E(P_n^{j-1})$ or $x\in y\in \mathcal{M}^j$. The new colour class is $A_{j+2}=\mathcal{M}^j$. 

By this construction, we obtain a triangle\-/free graph $G$ that is uniquely $k$\=/colourable with a colouring $\mathcal{C}=\{A_1,\dots,A_k\}$. 
It is $|A_1|=\Theta(n)$, $|A_2|=\Theta(n)$, $|A_3|=\Theta(n^3)$, $|A_4|=\Theta(n^8)$, \dots. 

Thus, sizes of colour classes are pairwise different (with a possible exception of $|A_1| = |A_2 |$);
moreover, the critical chromatic number tends to $k-1$ for $k\to \infty$.

\newparagraph
\refstepcounter{constr}
\need{2cm}
\noindent\textit{Construction~\theconstr\ (Bollobás, Sauer~\cite{bollobas1976uniquely}): }\label{constr3}

Bollobás and Sauer used a probabilistic approach to show the existence of uniquely $k$\=/colourable graphs 
with girth at least (a given constant) $g$. 
They started with  $k$\=/partite graphs, each partite set of size $n$ (assuming $n$ large enough) and $m=\binom{k}{2}n^{1+\varepsilon}$ uniformly chosen edges with $0<\varepsilon<\frac{1}{4g}$. 

The probability that such a graph contains only few cycles of length smaller than $g$ and that these cycles do not share a vertex tends to~1 (with $n\to\infty$). 
By removing a few edges to destroy all short cycles, the graph is asymptotically almost surely still uniquely $k$\=/colourable.
Thus, we obtain the existence of a demanded graph. 

Choosing $g=4$ and analysing arguments presented in~\cite{bollobas1976uniquely},  one can show
that there exists  a uniquely $k$\=/colourable triangle\-/free graph on $\Theta(k^{129})$ vertices. 
However, since the proof is probabilistic one, it yields no hint for
constructing such graphs.

\subsection*{Acknowledgement}
The author thanks Matthias Kriesell for his valuable suggestions. He further shows appreciation to the referees, whose constructive remarks helped to improve
the quality of this article.

\printbibliography

\end{document}